\documentclass{article}

\usepackage{arxiv}

\usepackage[utf8]{inputenc} 
\usepackage[T1]{fontenc}    
\usepackage{hyperref}       
\usepackage{url}            
\usepackage{booktabs}       
\usepackage{amsfonts}       
\usepackage{nicefrac}       
\usepackage{microtype}      
\usepackage{lipsum}

\usepackage{amssymb}
\usepackage{amsthm}

\usepackage{color}
\usepackage{url}
\usepackage{algorithm}
\usepackage{algpseudocode}
\algnotext{EndFor}
\algnotext{EndIf}
\usepackage{tikz}
\usetikzlibrary{arrows}
\usetikzlibrary{positioning}
\usepackage{tikz-qtree}
\usepackage{graphicx}
\usepackage{enumerate}
\usepackage{multirow}

\usepackage{url}

\usepackage{tikz-qtree}

\usepackage{algorithm}
\usepackage{algpseudocode}
\usepackage{algorithmicx}

\usepackage{amsmath,amssymb}

\newcommand{\msrepr}{\ensuremath{\operatorname{m}}}
\newcommand{\mrepr}{\ensuremath{\operatorname{r}}}

\newcommand{\dimms}{\ensuremath{\operatorname{\dim_{ms}}}}

\newcommand{\multiplicity}[2]{\ensuremath{#1[#2]}}

\newcommand{\ecc}{\ensuremath{\epsilon}}

\newcommand{\weak}{outer}

\newcommand{\Weak}{Outer}
\newcommand{\weakly}{outer}

\newcommand{\orderlessMetricRepresentation}{multiset representation} %

\newcommand{\orderlessMetricRepresentationPl}{multiset representations} %

\newcommand{\orderlessResolving}{multiset resolving set} %
\newcommand{\orderlessResolvingPl}{multiset resolving sets}

\newcommand{\orderlessMetricDimension}{\weak{} multiset dimension} %

\newcommand{\OrderlessMetricDimension}{\Weak{} Multiset Dimension} %

\newcommand{\orderlessBase}{\weak{} multiset basis} 
\newcommand{\orderlessBases}{\weak{} multiset bases} %

\newcommand{\msl}{\{\hspace*{-0.1cm}|}
\newcommand{\msr}{|\hspace*{-0.1cm}\}}

\newcommand{\multisetRel}{\cong_S}
\newcommand{\metricRel}{\sim_S}

\newcommand{\degree}{\ensuremath{\delta}}

\newcommand{\iso}{\ensuremath{\phi}}

\newcommand{\neighbour}{\ensuremath{\leftrightarrow}}

\newcommand{\nneighbour}{\ensuremath{\not \leftrightarrow}}

\newcommand{\decisionProblemName}{\textsc{DimMS}}

\newtheorem{theorem}{Theorem}[section]
\newtheorem{lemma}[theorem]{Lemma}
\newtheorem{corollary}[theorem]{Corollary}
\newtheorem{proposition}[theorem]{Proposition}

\newtheorem{definition}[theorem]{Definition}
\newtheorem{example}[theorem]{Example}

\newtheorem{remark}[theorem]{Remark}

\title{Distance-based vertex identification in graphs: the outer multiset
dimension}

\author{
   Reynaldo Gil-Pons \\
   Center for Pattern Recognition and Data Mining \\
   Patricio Lumumba s/n, Santiago de Cuba, 90500, Cuba \\
   \texttt{rey@cerpamid.co.cu} \\
   \And
   Yunior Ram\'irez-Cruz \\
   Interdisciplinary Centre for Security, Reliability and Trust \\ University 
   of Luxembourg \\
   6 av. de la Fonte, L-4364 Esch-sur-Alzette, Luxembourg \\
   \texttt{yunior.ramirez@uni.lu} \\
   \And
   Rolando Trujillo-Rasua \\
   School of Information Technology \\ Deakin University \\
   221 Burwood Hwy, Burwood, VIC 3125, Australia \\
   \texttt{rolando.trujillo@deakin.edu.au} \\
   \And
   Ismael G. Yero \\
   Departamento de Matem\'aticas, Escuela Polit\'ecnica Superior de Algeciras 
   \\ Universidad de C\'adiz\\
   Av. Ram\'on Puyol s/n, 11202 Algeciras, Spain\\
   \texttt{ismael.gonzalez@uca.es} \\
}

\begin{document}
\maketitle

\begin{abstract}
Given a graph $G$ and a subset of vertices $S = \{w_1, \ldots,
w_t\} \subseteq V(G)$, the \orderlessMetricRepresentation{}
of a vertex $u\in V(G)$ with respect to $S$ is the multiset
$m(u|S) = \msl d_G(u, w_1), \ldots, d_G(u, w_t) \msr$. A subset of
vertices $S$ such that $m(u|S) = m(v|S) \iff u = v$ for every $u, v \in V(G)
\setminus S$ is said to be a \orderlessResolving, and the cardinality of the
smallest such set is the \orderlessMetricDimension.
We study the general behaviour of the \orderlessMetricDimension,
and determine its exact value for several graph families.
We also show that computing the \orderlessMetricDimension{}
of arbitrary graphs is NP-hard, and provide methods for efficiently
handling particular cases.
\end{abstract}


\section{Introduction}\label{sec-intro}

The characterisation of vertices in a graph by means of unique features, known as
\emph{distinguishability} or \emph{resolvability}, has found
applications in computer
networks where nodes ought to be localised based on their properties rather
than on identifiers~\cite{khuller1996landmarks}, or to determine
the social role of an
actor in society in comparison to other peers with similar
structural properties~\cite{PhysRevE.73.026120}. In fact, simple structural
properties of vertices, such as their degree or
the subgraphs induced by their neighbours, have been successfully used to
re-identify (supposedly) anonymous users in social
graphs~\cite{LT2008,ZP2008,ZCO2009}.

This article focuses on vertex characterisations that are defined in relation
to a subset of vertices of the graph. The earliest of such characterisations is
known as \emph{metric representation}, introduced independently by
Slater~\cite{Slater1975} in $1975$ and Harary and Melter~\cite{Harary1976} in
$1976$. Formally,
given an ordered set of vertices $S = \{w_1, \ldots,
w_t\} \subseteq V$ in a graph $G = (V, E)$, the metric representation
of a vertex $u\in V$ with respect to $S$ is the $t$-vector $\mrepr(u|S) =
(d_G(u,
w_1), \ldots, d_G(u, w_t))$, where the metric $d_G(u, v)$
is computed as the length of a
shortest $u-v$ path in $G$. An ordered subset $S$ satisfying that every two
distinct vertices $u$ and $v$ in the graph have different metric
representation, i.e. $\mrepr(u|S) \neq \mrepr(v|S)$, is said to be a
\emph{resolving
set}. The minimum cardinality amongst the resolving sets in a graph $G$
is known as the \emph{metric dimension} of $G$, and denoted as~$\dim(G)$.
The metric dimension of graphs has been extensively studied in literature since
the 70s. Issues that are relevant to the present day, such as privacy in
online social networks, are still benefiting from such research
effort~\cite{MauwRamirezTrujillo2016,MRT2018,MauwTX16,Trujillo-Rasua2016}.

The assumption that resolvability requires an order to exist (or be imposed)
on a set $S$ for obtaining metric representations
remained unchallenged
until $2017$, when Simanjuntak, Vetr\'{i}k, and Mulia introduced the notion of
\emph{multiset representation}~\cite{2017arXiv171100225S}  by looking at the
multiset of distances rather than at the standard vector of distances.

For a vertex $u\in V$ and a vertex set $S\subseteq V$, the multiset
representation of $u$ with respect to $S$, denoted $\msrepr(u|S)$,
is defined by $$\msrepr(u|S)=\msl d_G(u, w_1), \ldots, d_G(u, w_t) \msr,$$
where $\msl . \msr$ denotes a multiset.

Using this definition, the notions
of resolvability in terms of the metric representation were straightforwardly
extended to consider resolvability in terms of the multiset
representation \cite{KhemmaniI18,2017arXiv171100225S}.
Our main observation in this article is that these straightforward extensions
are in fact an oversimplification of the problem of distinguishing vertices
in a graph based on the multiset representation.
We argue that
this problem
has two flavours, one of which has been neglected in literature.

\noindent \emph{Contributions.}
This article makes the following contributions. 

\begin{itemize}
	\item We generalise the metric dimension of graphs to accommodate
different characterisations of their vertices, such as the metric and multiset
representations (Section~\ref{sec-preliminaries}). We show that the metric
dimension problem with respect to the
multiset representation admits two interpretations: one that can be found
in the literature~\cite{KhemmaniI18,2017arXiv171100225S} and is known as the
\emph{multiset dimension}, and another one that we call
the \emph{\orderlessMetricDimension}. The latter is well-defined, 
whereas the multiset dimension~\cite{KhemmaniI18,2017arXiv171100225S}
is undefined for an infinite number of graphs. We also show that 
the \orderlessMetricDimension{}
finds applications on measuring the re-identification risk of users
in a social graph. To the best of our knowledge, the multiset dimension has no 
obvious practical application.
	\item We characterise several graph families
for which the \orderlessMetricDimension{} can be easily determined,
or bounded by the metric dimension (Section~\ref{sec-basic-results}).
	\item We prove that the problem of computing the \orderlessMetricDimension{}
	in a graph is NP-Hard (Section~\ref{sec-complexity}).
	\item We provide a polynomial computational procedure to calculate the
	\orderlessMetricDimension{} of full $2$-ary trees   
	(Section~\ref{sect-trees}), and a parallelisable 
	algorithm for the general case of full $\degree$-ary trees.
\end{itemize}



\section{A generalisation of the metric dimension}
\label{sec-preliminaries}

We consider a simple and connected graph $G=(V,E)$ where $V$
is a set of vertices and $E$ a set of edges. The distance
$d_{G}(v,u)$ between two vertices $v$ and $u$ in $G$ is the
number of edges in a shortest path connecting them. If there is no ambiguity,
we will simply write $d(v,u)$.

The metric dimension of graphs has traditionally been studied
based on the so-called metric representation, which is the vector of distances
from a vertex
to an ordered subset of vertices of the graph. To accommodate other types
of relations between vertices, we generalise the metric dimension by
considering any equivalence relation $\sim\, \subseteq V \times V$ over the set of
vertices of the graph. That is, we consider a relation $\sim$ that is
reflexive, symmetric, and transitive. We use $[u]_{\sim}$ to denote
the equivalence class of the vertex $u \in V$ with respect to the relation $\sim$,
while $V / \sim$ denotes the partition of $V$ composed of the equivalence
classes induced by $\sim$.

\begin{definition}[Resolving and \weakly{} resolving set]\label{def-resolving}
A subset $S$ of vertices in a graph $G = (V, E)$ is said to be
resolving (resp. \weakly{} resolving) with respect to $\sim$ if all
equivalence classes in
$V/\sim$ (resp. $(V-S)/\sim$ ) have cardinality one.
\end{definition}

While standard resolving sets distinguish all vertices in a graph,
outer resolving sets only look at those vertices that are not in $S$, hence the
name.
We remark that there exist applications working under the assumption
that $S$ is given, implying that vertices in $S$ do not
need to be distinguishable. For example, in an active re-identification attack
on a social
graph~\cite{BDK2007,Trujillo-Rasua2016}, attackers first retrieve a set of
attacker nodes by using a pattern recognition algorithm, then they re-identify
other users in the network based on their metric representations
with respect to the set of attacker nodes.

We use $\metricRel$ to denote the relation on the set of vertices of a graph
defined by $u \metricRel v \iff \mrepr(u |
S) = \mrepr(v | S)$, where $\mrepr(v | S)$ is the vector of distances from
$v$ to vertices in $S$, and $\multisetRel$ to denote the relation $u
\multisetRel v \iff \msrepr(u |
S) = \msrepr(v | S)$, where $\msrepr(v | S)$ is the multiset of distances from
$v$ to vertices in $S$. These two relations are interconnected in the following
way.

\begin{proposition}\label{prop-relations}
For every non-trivial graph $G$, the following facts hold:
\begin{enumerate}[i.]
\item Every resolving set of $G$ with respect to $\multisetRel$ is
an \weak{} resolving set.
\item Every \weak{} resolving set of $G$ with respect to $\multisetRel$ is
an \weak{} resolving set of~$G$ with respect to $\metricRel$.
\item Every \weak{} resolving set of~$G$ with respect to $\metricRel$ is
a resolving set of~$G$, and vice versa.
\end{enumerate}
\end{proposition}

\begin{proof}
Let $S\subseteq V(G)$ be a resolving set of $G$ with respect to $\multisetRel$.
Then, every pair of distinct vertices $u,v\in V(G)$
satisfy $\msrepr(u | S)\ne \msrepr(v | S)$. Thus, it trivially follows
that the same property holds for every pair of distinct vertices
$u,v\in V(G)\setminus S$. This completes the proof of (i).

The second property follows straightforwardly from the fact that
$\msrepr(u | S)\ne \msrepr(v | S) \Longrightarrow
\mrepr(u | S)\ne \mrepr(v | S)$, and (iii) is a well-known property of
resolving sets based on the metric representation. 
\end{proof}

Figure~\ref{fig-hierarchy} depicts the relations between resolvability notions
enunciated in Proposition~\ref{prop-relations} in the form of a hierarchy.
In the figure, every arrow from resolvability notion $A$
to resolvability notion $B$ indicates that a set $S$ which is resolving
as defined by $A$ is also resolving as defined by $B$. We use the following
shorthand notation in
Figure~\ref{fig-hierarchy} and in the remainder of this article.

\begin{itemize}
	\item \emph{resolving set} to denote a resolving set with
	respect to $\metricRel$.
	\item \emph{\orderlessResolving} to denote a  resolving set with
	respect to $\multisetRel$.
	\item \emph{\weakly{} resolving set} to denote an \weakly{}
	resolving set with
	respect to $\metricRel$.
	\item \emph{\weakly{} \orderlessResolving} to denote an \weakly{}
	resolving 	set with 	respect to $\multisetRel$.
\end{itemize}

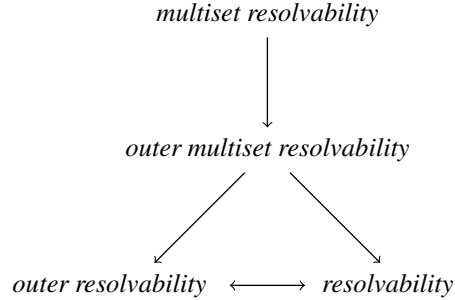
\begin{figure}[!ht]

\begin{center}

\begin{tikzpicture}[inner sep=0.7mm,
transition/.style={rectangle,draw=black!50,fill=black!20,thick},
line width=.5pt]

\coordinate (SMRS) at (0,1.5);
\coordinate (WMRN) at (0,0.3);
\coordinate (WMRSE) at (-0.3,-0.3);
\coordinate (WMRSW) at (0.3,-0.3);
\coordinate (WRN) at (-1.5,-1.5);
\coordinate (SRN) at (1.5,-1.5);
\coordinate (WRE) at (-0.5,-1.8);
\coordinate (SRW) at (0.5,-1.8);

\draw[black,->] (SMRS) -- (WMRN);
\draw[black,->] (WMRSE) -- (WRN);
\draw[black,->] (WMRSW) -- (SRN);
\draw[black,<->] (WRE) -- (SRW);

\coordinate [label=center:{\emph{multiset resolvability}}] (SMR) at (0,1.8);
\coordinate [label=center:{\emph{outer multiset resolvability}}] (WMR) at (0,0);
\coordinate [label=center:{\emph{outer resolvability}}] (WR) at (-2.1,-1.8);
\coordinate [label=center:{\emph{resolvability}}] (SR) at (1.6,-1.8);

\end{tikzpicture}

\end{center}

\caption{Hierarchy of resolvability notions.}\label{fig-hierarchy}
\end{figure}

\begin{definition}[Metric dimension and outer metric
dimension\label{def-dimension}]
The metric dimension (resp. outer metric dimension) of a simple
connected graph $G = (V, E)$ with respect to a structural relation $\sim$ is
the minimum cardinality amongst a resolving (resp. outer resolving) set in
$G$ with respect to $\sim$. If no resolving (resp. outer resolving) set
exists, we say that the metric dimension (resp. outer metric dimension)
is
undefined.
\end{definition}

An example of a metric dimension definition that is undefined for some
graphs is given by Simanjuntak et al.~\cite{2017arXiv171100225S}. They use the
multiset representation to distinguish vertices. It is easy to prove that
a complete graph has no multiset
resolving set, which leads to indefinition.
Conversely, the outer metric dimension with respect to the multiset
representation is always defined, given that for every graph $G = (E, V)$, $V$
is an outer multiset resolving set.

Overall, we highlight the fact that, while the outer metric dimension
and the standard metric dimension with respect to the metric representation
are equivalent (see Figure~\ref{fig-hierarchy}),
the use of the multiset
representation renders the outer metric dimension different from the standard
metric dimension. In fact, the outer multiset dimension is defined for any graph,
whereas the multiset dimension is not. Furthermore, recent privacy attacks
and countermeasures on social networks~\cite{BDK2007,Peng2012,Trujillo-Rasua2016}
rely on the notion of \emph{outer} resolving set,
rather than on the original notion of resolving set.
The remainder of this article is thus dedicated to the study
of the \emph{\orderlessMetricDimension},
that is, the \weakly{} metric dimension with respect to $\multisetRel$.



\section{Basic results on the \orderlessMetricDimension{}}
\label{sec-basic-results}

In this section we characterise several graph families
for which the \orderlessMetricDimension{} can be easily determined, or bounded
by the metric dimension otherwise. We start by providing notation that we use
throughout the paper.

\noindent \emph{Notation.}
Let $G=(V,E)$ be a graph of order $n=|V(G)|$. We will say that $G$
is \emph{non-trivial} if $n \ge 2$. $K_n$, $N_n$, $P_n$ and $C_n$ stand for the
complete, empty, path and cycle graphs, respectively, of order $n$.
Moreover, we will use the notation
$u \neighbour_{G} v$ (negated as $u \nneighbour_{G} v$) to indicate that $u$
and
$v$
are \emph{adjacent} in $G$, that is $(u,v)\in E$.
For a vertex $v$ of $G$, $N_G(v)$ denotes the set of neighbours of $v$ in $G$,
that is $N_G(v)=\{u\in V(G):\; u \neighbour v\}$. The set $N_G(v)$
is called the \emph{open neighbourhood} of the vertex $v$ in $G$
and $N_G[v]=N_G(v)\cup \{v\}$ is called the \emph{closed neighbourhood}
of $v$ in $G$. The degree of a vertex $v$ of $G$ will be denoted by
$\delta_G(v)$.
If there is no ambiguity, we will drop the subscripts and simply write
$u \neighbour v$, $u \nneighbour v$, $N(v)$, etc.
Two different vertices $u,v$ are called \emph{true twins} if $N[u] = N[v]$.
Likewise, $u,v$ are called \emph{false twins}
if $N(u) = N(v)$. In general, $u,v$ are called \emph{twins}
if they are either true twins or false twins. Moreover, a vertex $u$
is called a \emph{twin} if there exists $v\ne u$ such that $u$ and $v$
are twins. Note that the property of being twins induces an equivalence
relation on the vertex set of any graph. Finally, we will use the notation
$\dimms(G)$ for the
\orderlessMetricDimension{} of a graph $G$, and  $\dim(G)$ for the standard
metric dimension.



\begin{remark}\label{rm-general-bounds-msdim}
For every non-trivial graph $G$ of order $n$, the following facts hold:
\begin{enumerate}[i.]
\item $1\le\dimms(G)\le n-1$.
\item $\dimms(G)\ge\dim(G)$.
\end{enumerate}
\end{remark}

\begin{proof}
The fact that $\dimms(G)\ge 1$ follows directly from the definition
of \orderlessMetricDimension, whereas $\dimms(G)\le n-1$ follows trivially
from the fact that every vertex $v$ is the sole vertex
in $V(G)\setminus\left(V(G)\setminus\{v\}\right)$,
and thus it has a unique multiset representation w.r.t. $V(G)\setminus\{v\}$,
which is thus a \orderlessResolving.
The fact that $\dimms(G)\ge\dim(G)$ follows directly
from item (ii) of Proposition~\ref{prop-relations}.
\qed
\end{proof}

Once established the global bounds of the \orderlessMetricDimension{},
we now focus on the extreme cases of these inequalities.

\begin{remark}\label{rm-paths}
A graph $G$ satisfies $\dimms(G)=1$ if and only if  it is a path graph.
\end{remark}

\begin{proof}
Let $G$ be a path graph. It is clear that the set $\{ v\}$,
where $v$ is an extreme vertex of $G$,
is a \orderlessResolving{} of $G$, so $\dimms(G)\le 1$.
By item (ii) of Remark~\ref{rm-general-bounds-msdim}, $\dimms(G)\ge\dim(G)\ge 1$,
so the equality holds.
On the other hand, if $G$ is not a path graph, then item (ii)
of Remark~\ref{rm-general-bounds-msdim} also leads
to $\dimms(G)\ge\dim(G)\ge 2$, as the standard metric dimension of a graph
is known to be $1$ if and only if it is a path graph \cite{Chartrand2000}.
\end{proof}

According to Remark~\ref{rm-paths}, the cases where $\dimms(G)=\dim(G)=1$ coincide.
However, this is not the case for the upper bound
of Remark~\ref{rm-general-bounds-msdim} (i).
Indeed, while it is easy to see that, for any positive integer $n$,
the complete graph $K_n$ satisfies $\dimms(K_n)=\dim(K_n)=n-1$,
we have the fact that this is the sole family of graphs for which $\dim(K_n)=n-1$,
whereas there exist graphs $G$ such that $\dimms(G)=n-1>\dim(G)$,
as exemplified by the next results.

\begin{example}\label{ex-1}
The cycle graphs $C_4$ and $C_5$ satisfy $\dimms(C_4)=3>2=\dim(C_4)$
and $\dimms(C_5)=4>2=\dim(C_5)$.
\end{example}

\begin{remark}\label{rm-compl-k-partite}
Every complete $k$-partite graph
$G\cong K_{r_1,r_2,\ldots,r_k}$ such that
$r_1=r_2=\ldots=r_k\ge 2$ and $\sum_{i=1}^kr_i=n$ satisfies $\dimms(G)=n-1$.
\end{remark}

\begin{proof}
Let $G\cong K_{r_1,r_2,\ldots,r_k}$ be a complete $k$-partite
graph such that $r_1=r_2=\ldots=r_k\ge 2$. Let
$u,v\in V(G)$ be two arbitrary vertices of $G$ and let
$S\subseteq V(G)\setminus\{u,v\}$.
If $u \nneighbour v$, then $\msrepr(u\ |\ S)=\msrepr(v\ |\ S)$,
as they are false twins in~$G$. Consequently, $S$ is not
a \orderlessResolving{} of $G$. We now treat the case where $u \neighbour v$,
for which we differentiate the following subcases:

\begin{itemize}
\item $S=V(G)\setminus\{u,v\}$. In this case, $\msrepr(u\ |\ S)=\msrepr(v\ |\ S)=
\bigcup_{i=1}^{r-1}\msl 2 \msr \cup \bigcup_{i=1}^{r-1}\msl 1 \msr
\cup \bigcup_{i=1}^{k-2}\bigcup_{j=1}^{r}\msl 1 \msr$, and
so $S$ is not a \orderlessResolving{} of~$G$.
\item $S\subset V(G)\setminus\{u,v\}$. Here, if there exists some
$x\in V(G)\setminus(S\cup\{u,v\})$ such that $x \nneighbour u$ ($x \nneighbour
v$), then $\msrepr(u\ |\ S)=\msrepr(x\ |\ S)$
$\left(\msrepr(v\ |\ S)=\msrepr(x\ |\ S)\right)$,
as $x$ and $u$ ($x$ and $v$) are false twins in $G$. Thus, $S$ is not a \orderlessResolving{} of~$G$.
Finally, if every $x\in V(G)\setminus(S\cup\{u,v\})$ satisfies $u\neighbour
x\neighbour v$, then
we have that $\msrepr(u\ |\ S)=\msrepr(v\ |\ S)=
\bigcup_{i=1}^{r-1}\msl 2 \msr \cup \bigcup_{i=1}^{r-1}\msl 1 \msr
\cup \bigcup_{i=1}^{t_1}\msl 1 \msr \cup \ldots
\cup \bigcup_{i=1}^{t_{k-2}}\msl 1 \msr$,
with $t_i\le r$ for $i\in\{1,\ldots,k-2\}$, which entails
that $S$ is not a \orderlessResolving{} of~$G$.
\end{itemize}

Summing up the cases above, we have that no set
$S\subseteq V(G)$ such that $|S|\le n-2$
is a \orderlessResolving{} of $G$, and so $\dimms(G)\ge n-1$.
The equality follows from item (i) of Remark~\ref{rm-general-bounds-msdim}.
The proof is thus completed.
\end{proof}

%

Example~\ref{ex-1} shows two cases where the \orderlessMetricDimension{}
of a cycle graph is strictly larger than its standard metric dimension.
With the exception of $C_3$, which satisfies $\dimms(C_3)=\dim(C_3)=2$,
the strict inequality holds for every other cycle
graph, as shown by the following result.

\begin{remark}\label{rm-cycles}
Every cycle graph $C_n$ of order $n\ge 6$ satisfies $\dimms(C_n)=3$.
\end{remark}

\begin{proof}
Consider an arbitrary pair of vertices $u,v\in V(C_n)$ and a pair
of vertices $x,y\in V(C_n)\setminus\{u,v\}$ such that $ux\ldots yv$
is a path of $C_n$ (note that for $n\ge 6$ at least one such pair $x,y$ exists).
We have that $\msrepr(x\ |\ \{u,v\})=\msrepr(y\ |\ \{u,v\})
=\msl 1,d(u,v)\pm 1\msr$, so no vertex subset of
size 2 is a \orderlessResolving{} of $C_n$. Thus, $\dimms(C_n)\ge 3$.

Now, consider an arbitrary vertex $v_i\in V(C_n)$ and the set
$S=\{v_{i-2},v_i,v_{i+1}\}$, where the subscripts are taken modulo $n$.
We differentiate the following cases for a pair of vertices $x,y\in V(C_n)\setminus S$:
\begin{enumerate}
\item $x=v_{i-1}$. In this case, $\msrepr(x\ |\ S)=\msl 1,1,2\msr\neq\msrepr(y\ |\ S)$,
as $y$ is at distance 1 from at most one element in $S$.
\item $x$ and $y$ satisfy $\{d(x,v_i),d(x,v_{i-2})\}=\{d(y,v_i),d(y,v_{i-2})\}$.
In this case, assuming without loss of generality that $a=d(x,v_i)<d(y,v_i)$,
we have that $\msrepr(x\ |\ S)=\msl a, a+2,a-1\msr\neq\msl a,a+2,a+3\msr=\msrepr(y\ |\ S)$.
\item $x$ and $y$ satisfy $\{d(x,v_{i+1}),d(x,v_{i-2})\}=\{d(y,v_{i+1}),d(y,v_{i-2})\}$.
In a manner analogous to that of the previous case, we assume without loss of generality that $b=d(x,v_{i+1})<d(y,v_{i+1})$ and
obtain that $\msrepr(x\ |\ S)=\msl b,b+1,b+3\msr\neq\msl b,b+2,b+3\msr=\msrepr(y\ |\ S)$.
\item In every other case,
we have that $\min\{d\;|\;d\in\msrepr(x\ |\ S)\}\neq \min\{d'\;|\;d'\in\msrepr(y\ |\ S)$,
so $\msrepr(x\ |\ S)\neq \msrepr(y\ |\ S)$.
\end{enumerate}
Finally, summing up the cases above,
we have that $S$ is a \orderlessResolving{} of $G$, and so $\dimms(G)\le|S|=3$.
This completes the proof.
\end{proof}

%
%
%
%

Next, we characterise a large number of cases 
where
the outer multiset dimension is strictly greater than
the standard metric dimension. To that end, we first introduce some necessary
notation. We represent by $^{n}_{r}C_{rep}$ the number of $r$-combinations,
with repetition, from $n$ elements. Likewise, we represent
by $^{n}_{r}P_{rep}$ the number of $r$-permutations, with repetition,
from $n$ elements. Recall that $^{n}_{r}C_{rep}=\binom{r+n-1}{r}
=\binom{r+n-1}{n-1}$, whereas $^{n}_{r}P_{rep}=n^r$.
Finally, we recall the quantity $f(n,d)$, defined in~\cite{Chartrand2000}
as the smallest positive integer $k$ such that $k+d^k\ge n$. In an analogous
manner, we define $f'(n,d)$ as the smallest positive integer $k'$
such that $k'+\binom{r+d-1}{d-1}\ge n$. Since, by definition,
$^{n}_{r}C_{rep}\le ^{n}_{r}P_{rep}$, we have that $f(n,d)\le f'(n,d)$.
With the previous definitions in mind, we introduce our next result.

\begin{theorem}\label{th-dimms-strictly-greater}
For every graph $G=(V,E)$ of order $n$ and diameter $d$
such that $\dim(G)<f'(n,d)$,
$$\dimms(G)>\dim(G).$$
\end{theorem}

\begin{proof}
Let $G=(V,E)$ be a graph of order $n$ and diameter $d$.
It was proven in~\cite{Chartrand2000} that every such graph
satisfies $\dim(G)\ge f(n,d)$. Indeed, no vertex subset
$S\subseteq V$ such that $|S|<f(n,d)$ is a metric generator of $G$,
because the number of different metric representations,
with respect to $S$, for elements in $V\setminus S$
is at most $d^{|S|}<n-|S|=|V\setminus S|$.
In general, if $|S|=r$, the set of all possible different metric representations
for elements of $V\setminus S$ with respect
to $S$ is that of all permutations,
with repetition, of $r$ elements from $\{1,2,\ldots,d\}$.
Applying an analogous reasoning, we have that the set of all possible
different multiset metric representations for elements of $V\setminus S$
with respect to $S$
is that of all combinations, with repetition, of $r$ elements
from $\{1,2,\ldots,d\}$. Thus, any multiset metric generator $S$
of $G$ must satisfy $^{n}_{|S|}C_{rep}\ge n-|S|$, so $\dimms(G)\ge f'(n,d)$.
In consequence, if $\dim(G)<f'(n,d)$, then $\dimms(G)>\dim(G)$.
\end{proof}

An example of the previous result is the wheel graph
$W_{1,5}\cong \langle v\rangle+C_5$, which has diameter $2$
(see Figure~\ref{fig-wheel}). As discussed in~\cite{Chartrand2000,Hernando2005},
$\dim(W_{1,5})=2=f(6,2)<f'(6,2)=3<4=\dimms(W_{1,5})$.

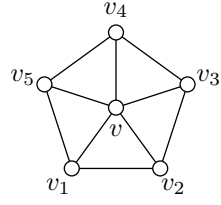
\begin{figure}[!ht]
\begin{center}

\begin{tikzpicture}[inner sep=0.7mm, place/.style={circle,draw=black,
fill=white},xx/.style={circle,draw=black!99,
fill=black!99},gray1/.style={circle,draw=black!99,
fill=black!25},gray2/.style={circle,draw=black!99,
fill=black!50},gray3/.style={circle,draw=black!99,fill=black!75},
transition/.style={rectangle,draw=black!50,fill=black!20,thick},
line width=.5pt]

\def\radius{1cm}
\def\lblradius{1.3cm}

\def\n{5}

\foreach \ind in {1,...,5}\pgfmathparse{90+360/\n*\ind}\coordinate
(g1\ind) at (\pgfmathresult:\radius);

\foreach \ind in {1,...,5}\pgfmathparse{90+360/\n*\ind}\coordinate
(v1\ind) at (\pgfmathresult:\lblradius);

\coordinate (v) at (0,0);

\draw[black] (g11) -- (g12) -- (g13) -- (g14) -- (g15) -- (g11);
\foreach \ind in {1,...,5} \draw[black] (v) -- (g1\ind);

\foreach \ind in {1,...,5} \node [place] at (g1\ind) {};

\node [place] at (v) {};

\node at (0,-0.3) {$v$};
\node at (v11) {$v_5$};
\node at (v12) {$v_1$};
\node at (v13) {$v_2$};
\node at (v14) {$v_3$};
\node at (v15) {$v_4$};

\end{tikzpicture}

\end{center}
\caption{The wheel graph $W_{1,5}\cong\langle v\rangle+C_5$.}
\label{fig-wheel}
\end{figure}

To conclude this section, we give a general result on the relation
between \weakly{} \orderlessResolvingPl{} and twin vertices,
a particular case of which will be useful in further sections of this paper.

\begin{proposition}\label{prop-twins}
Let $G$ be a non-trivial graph and let $S \subseteq V(G)$
be an \weakly{} \orderlessResolving{} of $G$.
Let $u,v\in V(G)$ be a pair of twin vertices.
Then, $u \in S$ or $v \in S$.
\end{proposition}

\begin{proof}
The proof follows from the fact that, as twin vertices, $u$ and $v$ satisfy
$d(u,x)=d(v,x)$ for every $x\in V(G)\setminus\{u,v\}$, which entails
that $u$ and $v$ have the same \orderlessMetricRepresentation{} according
to any subset of $V(G)\setminus\{u,v\}$.
\end{proof}

\begin{corollary}\label{cor-twin-equiv-classes}
Let $G$ be a non-trivial graph
and let $\mathcal{T}=\{\left[u_1\right], \left[u_2\right], \ldots,
\left[u_t\right]\}$ be the set of equivalence classes induced in $V(G)$
by the twin equivalence relation.
Then, $$\dimms(G)\ge
\sum_{i=1}^{t}\left(\left|\left[u_i\right]\right|-1\right).$$
\end{corollary}

\begin{proof}
The result follows from the fact that, for every twin equivalence class,
at most one element can be left out of any \weakly{} \orderlessResolving.
\end{proof}

\section{Complexity of the \orderlessMetricDimension{} problem}
\label{sec-complexity}

In the previous section, we showed that algorithms able to compute the metric 
dimension can be used to determined or bound the \orderlessMetricDimension. 
The trouble is, however, that calculating the metric dimension is 
NP-Hard~\cite{khuller1996landmarks}. We prove in this section that computing 
the \orderlessMetricDimension{} of a simple connected 
graph  is NP-hard as well.
The proof is, in some way, inspired by the NP-hardness proof of the 
metric dimension problem given in \cite{khuller1996landmarks}. 
To begin with, we formally state the decision problem associated 
to the computation of the \orderlessMetricDimension:

\medskip

\noindent
\textbf{\OrderlessMetricDimension{}} (\decisionProblemName)\\
INSTANCE: A graph $G=(V,E)$ and an integer $k$ satisfying $1 \le k \le |V|-1$.\\
QUESTION: Is $\dimms(G) \le k$?

\begin{theorem}
The problem \decisionProblemName{} is NP-complete.
\end{theorem}

\begin{proof}
The problem is clearly in NP. We give the NP-completeness proof by a reduction 
from 3-SAT. Consider an arbitrary input to 3-SAT, that is, a formula $F$ with 
$n$ variables and $m$ clauses. Let $x_1, x_2, \dots, x_n$ be the variables, and 
let $C_1, C_2, \dots, C_m$ be the clauses of $F$. We next construct a connected 
graph $G$ based on this formula $F$. To this end, we use the following gadgets.

For each variable $x_i$ we construct a gadget as follows (see Fig. 
\ref{fig:gadgetxi}).
\begin{itemize}
\item Nodes $T_i$, $F_i$ are the ``true'' and ``false'' ends of the gadget. The 
gadget is attached to
        the rest of the graph only through these nodes.
\item Nodes $a_i^1$, $a_i^2$, $b_i^1$, $b_i^2$ ``represent'' the value of the 
variable $x_i$, that is, $a_i^1$ and $a_i^2$ will be used to represent that 
variable $x_i$ is true, and $b_i^1$ and $b_i^2$ that it is false.
\item Nodes $d_i^1$ and $d_i^2$ will help to differentiate between nodes in 
different gadgets.
\item $Q_i$ is a set of end-nodes of cardinality $q_i$ adjacent to $d_i^1$. 
Notice that all these nodes are indistinguishable from $d_i^2$. Moreover, the 
cardinalities of these sets $Q_i$ are pairwise distinct, which is necessary for 
our purposes in the proof. We further on state the explicit values of their 
cardinalities.
\end{itemize}

\begin{figure}[!ht]
    \centering
        \begin{tikzpicture}
        \tikzset{main node/.style={circle,fill=white,draw,minimum 
        size=0.7cm,inner sep=0pt},}
        \tikzset{many node/.style={circle,fill=white,draw,minimum size=0.9cm},}
        \node[main node] (1) {$a_i^1$};
        \node[main node] (2) [right = 1cm of 1]  {$b_i^1$};
        \node[main node] (3) [below = 1cm of 1]  {$a_i^2$};
        \node[main node] (4) [below = 1cm of 2]  {$b_i^2$};
        \node[main node] (5) [below left = 2.5cm and 1.5cm of 1] {$T_i$};
        \node[main node] (6) [below right = 2.5cm and 1.5cm of 2] {$F_i$};
        \node[main node] (7) [below right = 0.2cm and 2.4cm of 5] {$d_i^1$};
        \node[main node] (8) [below left = 0.5cm and 0.5cm of 7] {$d_i^2$};
        \node[many node] (9) [below right = 0.5cm and 0.5cm of 7] {$Q_i$};

        \path[draw]
        (1) edge node {} (2)
        (3) edge node {} (4)
        (5) edge node {} (1)
        (2) edge node {} (6)
        (4) edge node {} (6)
        (3) edge node {} (5)
        (5) edge node {} (7)
        (6) edge node {} (7)
        (7) edge node {} (8);
        \path[draw,very thick]
        (7) edge node {} (9);
        \end{tikzpicture}
\caption{Gadget of a variable $x_i$}
\label{fig:gadgetxi}
\end{figure}
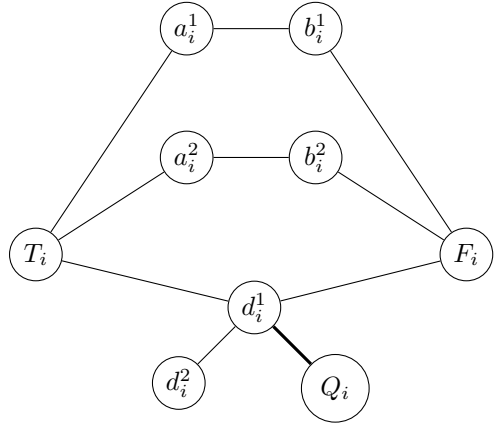
For each clause $C_j$ we construct a gadget as follows (see Figure 
\ref{fig:gadgetCj}).

\begin{itemize}
 \item Nodes $c_j^1$ and $c_j^3$ will be helpful in determining the truth 
 value of $C_j$.
 \item Nodes $c_j^2$ and $c_j^4$ will help to differentiate between nodes in 
 different gadgets.
 \item $P_j$ is a set of end-nodes of cardinality $p_j$ adjacent to $c_j^2$. 
 Notice that all these nodes are indistinguishable from $c_j^4$. As in the case 
 of the sets $Q_i$ from the variable gadgets, the cardinalities of these sets 
 $P_j$ are also pairwise distinct.
\end{itemize}

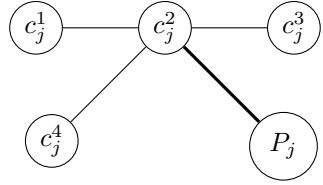
\begin{figure}[!ht]
    \centering
    \begin{tikzpicture}
    \tikzset{main node/.style={circle,fill=white,draw,minimum size=0.7cm,inner 
    sep=0pt},}
    \tikzset{many node/.style={circle,fill=white,draw,minimum size=0.9cm},}
    \node[main node] (1) {$c_j^1$};
    \node[main node] (2) [right = 1cm of 1]  {$c_j^2$};
    \node[main node] (3) [right = 1cm of 2]  {$c_j^3$};
    \node[main node] (4) [below left = 1cm and 1cm of 2] {$c_j^4$};
    \node[many node] (5) [below right = 1cm and 1cm of 2] {$P_j$};

    \path[draw]
    (1) edge node {} (2)
    (2) edge node {} (3)
    (2) edge node {} (4);
    \path[draw,very thick]
    (2) edge node {} (5);
    \end{tikzpicture}
    \caption{Gadget of clause $C_j$}
    \label{fig:gadgetCj}
    \end{figure}

As mentioned before, we require some conditions on the cardinalities of the 
sets $P_i$ and $Q_i$ from the variables and clauses gadgets, respectively. The 
values of their cardinalities (which we require in our proof) are as follows. 
For every $i\in \{1,\dots, n\}$ we make $q_i = 2 \cdot i \cdot n$, and for every 
$j\in \{1,\dots, m\}$ we make $p_j = 2 \cdot j \cdot n + 2 n ^ 2$. In 
concordance, we notice that the set of numbers $p_i$ and $q_j$ are pairwise 
distinct. Also, we clearly see that $\sum q_i + \sum p_j$ 
is polynomial in $n + m$. 

The gadgets representing the variables and the gadgets representing the clauses 
are connected in the following way in order to construct our graph $G$.

    \begin{itemize}
        \item Nodes $c_j^1$, for every $j$, are adjacent to nodes $T_i$, $F_i$ 
        for all $i$.
        \item If a variable $x_i$ does not appear in a clause $C_j$, then the 
        nodes $T_i$, $F_i$ are adjacent to $c_j^3$.
        \item If a variable $x_i$ appears as a positive literal in a clause 
        $C_j$, then the node $F_i$ is adjacent to $c_j^3$.
        \item If a variable $x_i$ appears as a negative literal in a clause 
        $C_j$, then the node $T_i$ is adjacent to $c_j^3$.
    \end{itemize}

We first remark that the constructed graph $G$ is connected, and that its order 
is polynomial in the number of variables and clauses of the original 
3-SAT instance. We will prove now that the formula $F$ is satisfiable if and 
only if the multiset dimension of $G$ is exactly $M=\sum_{i=1}^n q_i + 
\sum_{j=1}^m p_j + n$.

First, let us look at some properties that must be fulfilled 
by a \orderlessResolving{} $S$ of minimum cardinality in $G$. 
First, as the nodes in 
$Q_i\cup \{d_i^2\}$, for every $i\in\{1,\dots,n\}$, are indistinguishable among 
them, and at least $|Q_i|$ of them must be in $S$, we can assume without lost 
of generality that $Q_i\subset S$. By using a similar reasoning, also 
$P_j\subset S$ for every $j\in\{1,\dots,m\}$. Moreover, for every 
$i\in\{1,\dots,n\}$, at least one of the nodes $a_i^1$, $a_i^2$, $b_i^1$, 
$b_i^2$ must be in $S$, otherwise some pairs of them would have the same 
\orderlessMetricRepresentation, which is not possible. 
Thus, the cardinality of $S$ is 
at least $M$. Clearly, if $M=|S|$, then we have already fully described a set 
of nodes that could represent $S$.

\begin{lemma}\label{keylemmaNP1}
Consider a set $S^*$ containing exactly $M$ nodes given as follows. All nodes 
in $Q_i$ for $i \in \{1, \ldots,
n\}$, all nodes in $P_j$ for $j \in \{1, \ldots,  m\}$, and exactly one node 
from each set $\{a_i^1, a_i^2, b_i^1 , b_i^2\}$ for $i \in \{1, \ldots,
 n\}$ are in $S^*$. Then, all pairs of nodes have different \orderlessMetricRepresentationPl{} with respect to $S^*$, except possibly $c_j^1$ and $c_j^3$ 
 $($for some $j\in \{1, \ldots,  m\}$$)$.
 \end{lemma}

\begin{proof}
To prove the lemma, we will explicitly compute the \orderlessMetricRepresentation{} of 
each node. For easier representation, we use a vector $(x_1, \ldots, x_n)$ to 
denote the multiset over positive integers such that $1$ has multiplicity 
$x_1$, $2$ has multiplicity $x_2$, and so on. 
        \begin{itemize}
            \item $\msrepr( c_j^4 |  S^* ) = (0, p_j, 0, n, \cdots)$ 
            \item $\msrepr( c_j^2 | S^* ) = (p_j, \cdots)$
            \item $\msrepr(d_i^1 | S^* ) = (q_i, \cdots)$
            \item $\msrepr(d_i^2 | S^* ) = (0, q_i, \cdots)$
            \item $(\msrepr(T_i | S^* ), \msrepr( F_i| S^* ))$ is equal 
            to either  $((1, q_i, \cdots), (0, q_i + 1, 
            \cdots))$ or $((0, q_i + 1, \cdots), (1, q_i, \cdots))$
            \item $a_i^1$, $a_i^2$, $b_i^1$, $b_i^2$: Assume $b_i^2 \in S^*$. Then $\{ \msrepr(a_i^1 | S^* ),  \msrepr(a_i^2 | S^* ), 
            \msrepr(b_i^1 | S^* )\} = \{(1, 
			0, q_i, \cdots),(0, 1, q_i, \cdots),(0, 0, q_i + 1, 
			\cdots)\}$. An analogous result remains if the assumption that $b_i^2 \in S^*$ is dropped, based on the following 
observations. First, one and only one of the nodes $a_i^1, a_i^2, b_i^1, 
b_i^2$ is in $S^*$, and the distances from the other three to this one are
            exactly $1, 2, 3$ in some order.
%
			Second, each of these nodes have $q_i$ nodes at distance 3. 
            \item $\msrepr(c_j^1 | S^* ) = (0, p_j + n, \sum_{i=1}^n q_i,  \sum_{l=1}^m p_l - 
            p_j)$
            \item $c_j^3 : $ the number of nodes at distance two depends on 
            which node belongs to $S^*$ from each variable gadget. We 
            distinguish three possible cases for the distance between $c_j^3$ 
            and the node from $S^*$ belonging to the gadget corresponding to a 
            variable $x_i$.
                \begin{itemize}
                    \item If $x_i$ appears in $C_j$ as a positive literal, and 
                    $a_i^1\in S^*$ or $a_i ^2\in S^*$, then such distance is 3.
                    \item If $x_i$ appears in $C_j$ as a negative literal, and 
                    $b_i^1\in S^*$ or $b_i ^2\in S^*$, then such distance is 3.
                    \item If none of the above situations occurs, then such 
                    distance is 2.
                \end{itemize}
			Therefore, the \orderlessMetricRepresentation{} of $c_j^3$ is related to the 
			set $(0, p_j +
			w_j,\sum q_i + n - w_j, \sum p_l - p_j)$, where $w_j$ is the 
			number of nodes
			from gadgets representing some $x_i$ matching the third case above. 
        \end{itemize}
        Notice that, as the difference between any $p_j$ and any $q_i$ is at 
        least $2n$, all pairs of nodes have also a different \orderlessMetricRepresentation{}, except possibly $(c_j^1, c_j^3)$ that depend on the 
        selected nodes from each variable gadget. We next particularise some of 
        these situations.
        \begin{itemize}
            \item As $q_i \neq p_j$ for every $i \in \{1, \ldots, n\}$ and 
            every $j \in \{1, \ldots, m\}$, we observe $\msrepr ( c_j^2 | S^*) 
            \neq \msrepr ( d_i^1 | S^*)$, $\msrepr ( c_j^4 | S^*) \neq \msrepr ( 
            d_i^2 | S^*)$ .
            \item Since $p_j \neq q_i+1$ for every $i \in \{1, \ldots,
            n\}$ and every $j \in \{1, \ldots, m\}$, we deduce $\msrepr ( c_j^4 
            | S^*) \neq \msrepr ( T_i | S^*)$ and $\msrepr ( c_j^4 | S^*)  \neq 
            \msrepr ( F_i | S^*)$.
            \item Since $p_{j_1} \neq p_{j_2}+n$ for every $j_1, j_2 \in \{1, 
            \ldots, m\}$, we get $\msrepr ( c_{j_1}^4 | S^*) \neq \msrepr ( 
            c_{j_2}^1 | S^*)$.

        \end{itemize}
Remaining cases trivially follow, and are left to the reader, and so the proof of 
the lemma is complete.  \end{proof}

We will now show a way to transform the set $S^*$ into values for the variables 
$x_i$ that will lead to a satisfiable assignment for $F$. If $S^*\cap 
\{a_i^1,a_i^2\}\ne \emptyset$ for some variable $x_i$, then we set the variable 
$x_i=\mathtt{true}$ (with respect to $S^*$). Otherwise ($S^*\cap 
\{a_i^1,a_i^2\}=\emptyset$ or equivalently $S^*\cap \{b_i^1,b_i^2\}\ne 
\emptyset$), we set $x_i=\mathtt{false}$. Hence, the clause $C_j$ is 
$\mathtt{true}$ or $\mathtt{false}$ in the natural way, according to the values 
previously given to its variables.

\begin{lemma}\label{keylemmaNP2}
Let $S^*$ be a set of nodes as defined in the premise of Lemma~\ref{keylemmaNP1}. 
Then $c_j^1$ and $c_j^3$ have 
different \orderlessMetricRepresentationPl{} with respect to $S^*$ 
if and only if the clause $C_j$ is $\mathtt{true}$. 
\end{lemma}

\begin{proof}
Notice that the distance between $c_j^3$ and the node in $S^*$ from the gadget 
corresponding to $x_i$ is 3 if and only if the clause $C_j$ is $\mathtt{true}$ 
(see Lemma \ref{keylemmaNP1}). Thus, $w_j = 0$ (as defined in Lemma 
\ref{keylemmaNP1}) when the clause $C_j$ is $\mathtt{false}$, and only in this 
case  $\msrepr ( c_j^3 | S^*) = \msrepr(c_j^1 | S^*)$.
\end{proof}

By using the lemmas above, we conclude the NP-completeness reduction, through 
the following two lemmas.

\begin{lemma}
If $F$ is satisfiable, then the \orderlessMetricDimension{} of $G$ is $M$.  
\end{lemma}

\begin{proof}
Recall that $\dimms(G)\ge M$. It remains to prove that if $F$ is satisfiable then $\dimms(G)\le M$. Let us construct a set $S$ in the following way. If $x_i$ is $\mathtt{true}$, 
then $a_i^1 \in S$. Otherwise ($x_i$ is $\mathtt{false}$), $b_i^1 \in S$. Also, 
we add to $S$ all nodes in the sets $P_j$ and $Q_i$. Hence, according to Lemmas 
\ref{keylemmaNP1} and \ref{keylemmaNP2}, $S$ is a \orderlessResolving, 
and its cardinality is exactly $M$.
\end{proof}

\begin{lemma}
If the \orderlessMetricDimension{} of $G$ is $M$, then $F$ is satisfiable. 
\end{lemma}

\begin{proof}
Let $S$ be a set of nodes of cardinality equal to the multiset dimension of 
$G$. Hence, as explained before, without lost of generality all nodes in the 
sets $P_j$, $Q_i$, and exactly one node of  $a_i^1$, $a_i^2$, $b_i^1$, $b_i^2$, 
must belong to $S$, and no other node is in $S$. If $a_i^1\in S$ or $a_i^2\in 
S$, then let $x_i$ be $\mathtt{true}$. Otherwise, let $x_i$ be 
$\mathtt{false}$. Since $S$ is a \orderlessResolving, 
according to Lemmas \ref{keylemmaNP1} and \ref{keylemmaNP2}, all clauses $C_j$ 
of $F$ must be $\mathtt{true}$, unless the nodes $c_j^1$ and $c_j^3$ would have 
the same \orderlessMetricRepresentation{}, which is not possible. If all clauses of $F$ 
are $\mathtt{true}$, then $F$ is satisfiable, as claimed.
\end{proof}
The last two lemmas together complete the reduction from 3-SAT to the problem 
of deciding whether the \orderlessMetricDimension{} of a graph $G$ is equal 
to a given positive integer. The latter problem can in turn 
be trivially reduced to \decisionProblemName. This completes the proof. 
\end{proof}

\section{Particular cases involving trees}\label{sect-trees}

Given that, in general, computing the \orderlessMetricDimension{} of a graph 
is NP-hard, it remains an open question for which families of graphs 
the \orderlessMetricDimension{} can be 
efficiently computed. The goal of this section is to provide a computational 
procedure and a closed formula to compute the 
\orderlessMetricDimension{} of full $\degree$-ary 
trees. A \emph{full $\degree$-ary tree} is a rooted tree 
whose root has degree~$\degree$, all its leaves are at the same distance from 
the root, and its descendants are either leaves or vertices of degree 
$\degree+1$. We expect the results obtained in this section to pave the way 
for the study of the \orderlessMetricDimension{} of general trees. 

\noindent
\emph{Notation.}
Given a multiset $M$ and an element $x$, we denote the multiplicity 
	of $x$ in $M$ as $\multiplicity{M}{x}$. We use $\ecc_G(x)$ to denote the 
	eccentricity of the vertex $x$ in a 
	graph $G$, which is defined
	as the largest distance between $x$ and any other vertex in the graph. We 
	will simply write $\ecc(x)$ if the considered graph is clear
	from the context.
Given a tree $T$ rooted in $w$, we use $T_x$ to denote the subtree
	induced by $x$ and all descendants of $x$, \emph{i.e.} those vertices 
	having a shortest path to $w$ that contains~$x$.	
Finally, an 
	\orderlessBase{} is said to be an outer multiset resolving set of minimum 
	cardinality.

%

We start by enunciating a simple lemma that
characterises \orderlessResolvingPl{} in full $\degree$-ary trees.

\begin{lemma}\label{theo-tree-main1}
Let $T$ be a full $\degree$-ary tree rooted in $w$ with $\degree > 1$. A set of
vertices $S \subseteq V(T)$ is an outer multiset resolving set if and only if
$\forall_{u, v \in V(T)\setminus S}\colon d(u, w) = d(v, w) \implies \msrepr(u|S) \neq
\msrepr(v|S)$.
\end{lemma}

\begin{proof}
Necessity follows from the definition of \weak{} \orderlessResolvingPl. To 
prove 
sufficiency we need to prove that
\begin{equation*}\label{eq2-theo-tree-main1}
\forall_{u, v
\in V(T)\setminus S}\colon d(u, w) \neq d(v, w) \implies \msrepr(u|S) \neq
\msrepr(v|S) \text{.}
\end{equation*}


Take two
vertices $x, y \in V(T)\setminus S$ such that $d(x, w) < d(y, w)$. Because $T$
is a full $\degree$-ary tree, we obtain that $d(x, w) <
d(y, w) \iff \ecc_{T}(x) < \ecc_{T}(y)$. Also, there must exist two 
leaf vertices $y_1, y_2$ in $T$ which are siblings 
and satisfy $d(y_1, y) = d(y_2, y)= \ecc_{T}(y)$. 
Considering that $y_1$ and $y_2$ are false twins, we obtain
that $
y_1, y_2 \not \in S \implies \msrepr(y_1|S) = \msrepr(y_2 | S)$. Therefore,
given that $d(y_1, w) =
d(y_2, w)$, it follows that $y_1 \in S$ or $y_2 \in S$.
We assume, without loss of generality, that $y_1 \in S$. 
On the one hand, we have that $d(y_1,y) \in \msrepr(y | S)$. 
On the other hand, because $d(y_1, y) = \ecc_{T}(y) > \ecc_{T}(x)$, 
we obtain that $d(y_1, y) \notin \msrepr(x | S)$, 
implying that $\msrepr(x | S) \neq \msrepr(y | S)$. 
\end{proof}

Based on the result above, we provide conditions under which 
an \orderlessBase{} can be constructed in a recursive manner. Recall that an 
\orderlessBase{} is an outer multiset resolving set of minimum 
cardinality. 


\begin{lemma}\label{theo-recursive-basis1}
Given a natural number $\ell > 1$, let $T_1, \ldots, T_{\degree}$ be $\degree$
full
$\degree$-ary
trees of depth $\ell$ with pairwise disjoint vertex sets. 
Let $w_1, \ldots, w_\degree$
be the roots of $T_1, \ldots, T_\degree$, respectively, and let $T$ be the full
$\degree$-ary tree rooted in $w$ defined by
the set of vertices $V(T) = V(T_1) \cup \cdots \cup V(T_\degree) \cup \{w\}$
and edges $E(T) =
E(T_1) \cup \cdots \cup E(T_\degree) \cup \{(w, w_1), \ldots, (w,
w_\degree)\}$. Let $S_1, \ldots, S_\degree$ be \orderlessBases{} of $T_1, \ldots, T_\degree$,
respectively. Then
\begin{align*}
& \forall_{i \neq j \in \{1, \ldots, \degree\}} \multiplicity{\msrepr_{T_i}(w_i
|
S_i)}{\ecc_{T_i}(w_i)} \neq \multiplicity{\msrepr_{Tj}(w_j |
S_j)}{\ecc_{T_j}(w_j)} \implies \\
& \quad \quad S_1 \cup \ldots \cup S_\degree \text{ is an
\orderlessBase{} of } T \text{.}
\end{align*}
\end{lemma}

\begin{proof}
Consider two vertices $x$ and $y$ in $T$ such that $d_{T}(x, w) = d_{T}(y, w)$.
We will prove that $\msrepr_{T}(x|S) \neq \msrepr_{T}(y|S)$, which gives that 
$S$ is an \weak{}
\orderlessResolving{} via application of Lemma~\ref{theo-tree-main1}. Our
proof is split in two cases, depending on whether $x$ and $y$ are within the
same sub-branch or not.

First, assume that $x \in V(T_i)$ and $y \in V(T_j)$ for some $i \neq j \in
\{1, \ldots, \degree\}$. For every leaf vertex $z$ in $T$, but not in $T_i$, we
obtain that $d_T(x, z) = \ecc_{T}(x)$. Because $\ecc_{T}(x) > \ecc_{T_i}(x)$,
we get $\msrepr_{T}(x|S)[\ecc_{T}(x)] = \sum_{k \in \{1, \ldots, \degree\}
\setminus \{i\}}
\msrepr_{T_k }(w_k|S)[\ecc_{T_k}(w_k)]$.
Analogously, we obtain that $\msrepr_{T}(y|S)[\ecc_{T}(y)] = \sum_{k \in \{1,
\ldots, \degree\}
\setminus \{j\}}
\msrepr_{T_k }(w_k|S)[\ecc_{T_k}(w_k)]$. Therefore,
\[
\msrepr_{T}(x|S)[\ecc_{T}(x)] - \msrepr_{T}(y|S)[\ecc_{T}(y)] = \msrepr_{T_j
}(w_j|S)[\ecc_{T_j}(w_j)] - \msrepr_{T_i }(w_i|S)[\ecc_{T_i}(w_i)] \text{.}
\]

By considering the fact that $\msrepr_{T_j
}(w_j|S)[\ecc_{T_j}(w_j)] \neq \msrepr_{T_i}(w_i|S)[\ecc_{T_i}(w_i)]$, we
obtain that $\msrepr_{T}(x|S)[\ecc_{T}(x)] \neq
\msrepr_{T}(y|S)[\ecc_{T}(y)]$, which implies that  $\msrepr_{T}(x|S)\neq
 \msrepr_{T}(y|S)$.

For the second case assume that $x \in V(T_i)$ and $y \in V(T_i)$ for some $i
\in \{1, \ldots, \degree\}$. This implies
that
$\msrepr_{T}(x|S_i) \neq \msrepr_{T}(y|S_i)$, because $S_i$ is a
\orderlessResolving{} in $T_i$. Moreover, for every vertex $z \in V(T)
\setminus V(T_i)$ it holds that $d_{T}(x, z) = d_{T}(y, z)$, which gives the
expected result: $\msrepr_{T}(x|S) \neq \msrepr_{T}(y|S)$.

Up to here we have proved that $S$ is an \weak{} \orderlessResolving{}. To 
prove that $S$
is a basis, we only need to show that for any \weak{} \orderlessResolving{} $S'$ 
in $T$,
it is satisfied that the sets $S' \cap V(T_1), \ldots, S' \cap V(T_\degree)$ are
\weak{} \orderlessResolvingPl{} in $T_1, \ldots, T_{\degree}$, respectively. 
Given that $S_1, \ldots,
S_{\degree}$ are \orderlessBases, this would mean that $S$ is an \weak{} 
\orderlessResolving{} of minimum
cardinality.

We proceed by contrapositive. Let $S_1' = S' \cap V(T_1), \ldots, S_{\degree}'
= S' \cap V(T_2)$. Assume that $S_i'$ is not an \weak{} \orderlessResolving{} 
in $T_i$
for some $i \in \{1,\ldots, \degree\}$.
Then, there must exist vertices $x$ and $y$ such that $\msrepr_{T_i}(x | S_i') = \msrepr_{T_i}(y |
S_i')$. As in a previous reasoning, since $x$ and $y$ are both within $T_i$, it
follows that
$\forall_{z
\in V(T) \setminus V(T_i)} d_T(x, z) = d_T(y, z)$. Hence, $\msrepr_{T}(x | S')
= \msrepr_{T}(y | S')$, which is a contradiction.
\end{proof}

Lemma~\ref{theo-recursive-basis1} provides a sufficient condition for
obtaining an \orderlessBase{} of a full
$\degree$-ary tree $T$ by joining bases of the first level branches of
$T$. This is useful for the development of a computational
procedure that finds the \orderlessMetricDimension{} of an arbitrary full
$\degree$-ary tree. Despite this fact, here we are interested in finding 
a closed formula for the \orderlessMetricDimension{} of full 
$\degree$-ary trees. The next result will prove itself a key element 
towards such a goal. 

\begin{theorem}\label{theo-main-tree}
Let $T_{\ell}^{\degree}$ be a full $\degree$-ary tree of depth $\ell$.
Let $n$ be the smallest positive integer such that there exist 
$\degree+1$ \orderlessBases{} $S_1, \ldots, S_{\degree+1}$ in $T_{n}^{\degree}$ 
satisfying that $\forall_{i \neq j \in \{1,\ldots,{\degree}+1\}} 
\multiplicity{\msrepr_{T_{n}^{\degree}}(w |S_i)}{\ecc_{T_{n}^{\degree}}(w)} 
\neq \multiplicity{\msrepr_{T_{n}^{\degree}}(w|S_j)}{\ecc_{T_{n}^{\degree}}(w)}$,  
where $w$ is the root of $T_{n}^{\degree}$. Then, for every $\ell \geq n$, the
\orderlessMetricDimension{} of $T^{\degree}_{\ell}$ 
is given by $\degree^{\ell-n} \times \dimms(T^{\degree}_{n})$. 
\end{theorem}

\begin{proof}
We proceed by induction.

\noindent
\emph{Hypothesis. } For some $\ell \geq n$, the following two conditions hold: 

\begin{enumerate}
	\item \label{cond1} There exists $\degree+1$ \orderlessBases{} $S_1,
\ldots,
S_{\degree+1}$ in $T_{\ell}^{\degree}$ satisfying that $\forall_{i \neq j \in
\{1, \ldots, {\degree}+1\}} \multiplicity{\msrepr_{T_{\ell}^{\degree}}(w |
S_i)}{\ecc_{T_{\ell}^{\degree}}(w)} \neq
\multiplicity{\msrepr_{T_{\ell}^{\degree}}(w
| S_j)}{\ecc_{T_{\ell}^{\degree}}(w)}$
	\item The
\orderlessMetricDimension{} of $T^{\degree}_{\ell}$ is given by
$\degree^{\ell-n} \times
\dimms(T^{\degree}_{n})$.
\end{enumerate}

Clearly, these two conditions hold for $\ell = n$ (base case). The remainder of
this proof will be dedicated to finding $\degree+1$
\orderlessBases{} $R_1,\ldots, R_{\degree+1}$ of $T_{\ell+1}^{\degree}$ that satisfy
condition~(\ref{cond1}). The second condition will follow straightforwardly 
from the size of the bases $R_1,
\ldots, R_{\degree+1}$.

Let $w'$ be the root of $T^{\degree}_{\ell+1}$ and $w$ the root of
$T^{\degree}_{\ell}$. Let
$w_1, \ldots, w_{\degree}$ be the children vertices of $w'$ in
$T^{\degree}_{\ell+1}$. For each sub-branch $T_{w_{k}}$ of
$T^{\degree}_{\ell+1}$, with $k \in \{1,\ldots,\degree\}$, 
let $\iso_k$ be an isomorphism from $T_\ell^{\degree}$ to
$T_{w_{k}}$. It follows that
$\hat{S}_k = \{\iso_k(u) | u \in S_k \}$ is an
\orderlessBase{} of $T_{w_{k}}$, for every $k \in \{1, \ldots, \degree +1\}$.
Moreover, given that $\forall_{k \in \{1, \ldots, \degree
+1\}} \msrepr_{T_{w_k}}(w_k
| \hat{S}_k) = \msrepr_{T^{\degree}_{\ell}}(w
| S_k)$, we conclude that
\begin{align*}
& \forall_{i \neq j \in \{1, \ldots,
\degree+1\}} \multiplicity{\msrepr_{T_{w_i}}(w_i |
\hat{S}_i)}{\ecc_{T_{w_i}}(w_i)} \neq
\multiplicity{\msrepr_{T_{w_j}}(w_j
| \hat{S}_j)}{\ecc_{T_{w_j}}(w_j)}
\end{align*}

By Theorem~\ref{theo-recursive-basis1}, we obtain that, for every $i \in \{1,
\ldots, \degree+1\}$, the set $R_i = \bigcup_{j \in \{1, \ldots,
\degree+1\}\setminus \{i\}} \hat{S}_j$
is an \orderlessBase{} of $T^{\degree}_{\ell+1}$.
Moreover, for
every $i \in \{1, \ldots, \degree+1\}$, the following holds
\begin{align*}
& \multiplicity{\msrepr_{T_{\ell+1}^{\degree}}(w' |
R_i)}{\ecc_{T_{\ell+1}^{\degree}}(w')} = \sum_{j \in \{1, \ldots,
\degree+1\} \setminus \{i\}} \multiplicity{\msrepr_{T_{w_j}}(w_j |
\hat{S}_i)}{\ecc_{T_{w_j}}(w_j)}
\end{align*}
From the equation above we obtain that for every $i, j \in \{1, \ldots,
\degree+1\}$,
\begin{align*}
& \multiplicity{\msrepr_{T_{\ell+1}^{\degree}}(w |
R_i)}{\ecc_{T_{\ell+1}^{\degree}}(w)} -
\multiplicity{\msrepr_{T_{\ell+1}^{\degree}}(w |
R_j)}{\ecc_{T_{\ell+1}^{\degree}}(w)} = \\
& \quad \quad \multiplicity{\msrepr_{T_{w_j}}(w_j |
\hat{S}_j)}{\ecc_{T_{w_j}}(w_j)} - \multiplicity{\msrepr_{T_{w_i}}(w_i |
\hat{S}_i)}{\ecc_{T_{w_i}}(w_i)}
\end{align*}
Recall that $\forall_{k \in \{1, \ldots, \degree
+1\}} \msrepr_{T_{w_k}}(w_k
| \hat{S}_k) = \msrepr_{T^{\degree}_{\ell}}(w
| S_k)$, which means that $i \neq j \implies
\multiplicity{\msrepr_{T_{w_j}}(w_j |
\hat{S}_j)}{\ecc_{T_{w_j}}(w_j)} \neq \multiplicity{\msrepr_{T_{w_i}}(w_i |
\hat{S}_i)}{\ecc_{T_{w_i}}(w_i)}$. Therefore, we conclude that
$T_{\ell+1}^{\degree}$ and $R_1, \ldots,
R_{\degree+1}$ satisfy the first condition of the induction hypothesis, 
\emph{i.e.}
\[
\forall_{i \neq j \in \{1,
\ldots,
{\degree}+1\}} \multiplicity{\msrepr_{T_{\ell+1}^{\degree}}(w |
R_i)}{\ecc_{T_{\ell+1}^{\degree}}(w)} \neq
\multiplicity{\msrepr_{T_{\ell+1}^{\degree}}(w
| R_j)}{\ecc_{T_{\ell+1}^{\degree}}(w)}
\]
Finally, observe that $|R_1| = |\hat{S}_2| \times \cdots\times |\hat{S}_{\degree+1}|
= \degree \times \dimms(T^{\degree}_{\ell})$. The second condition of the induction
hypothesis states that
$\dimms(T^{\degree}_{\ell}) = \degree^{\ell-n} \times \dimms(T^{\degree}_{n})$,
which gives that $\dimms(T^{\degree}_{\ell+1}) =
\degree \times \dimms(T^{\degree}_{\ell}) = \degree^{\ell+1-n} \times
\dimms(T^{\degree}_{n})$.
\end{proof}

We end this section by addressing the problem of finding the smallest $n$ such
that $T_{n}^{\degree}$ contains $\degree+1$ \orderlessBases{} $S_1,
\ldots, S_{\degree+1}$ satisfying the premises of Theorem~\ref{theo-main-tree}.
We do so by developing a computer program\footnote{The computer program can be
found at \url{https://github.com/rolandotr/graph}.} that calculates such number
via exhaustive search. The 
pseudocode for this 
computer program
can be found in 
Algorithm~\ref{alg-main}. 
It reduces the 
search space by bounding the size of an \orderlessBase{} with the help of 
Lemma~\ref{theo-recursive-basis1} (see Step~\ref{step-bound} of 
Algorithm~\ref{alg-main}). That said, we cannot guarantee termination of 
Algorithm~\ref{alg-main}, essentially for two reasons. First, the computational 
complexity 
of each iteration of the algorithm is exponential on the size of 
$T_n^{\degree}$ while, at the same time, the size of $T_n^{\degree}$ 
exponentially increases with $n$. Second, there is no theoretical guarantees 
that such an $n$ can be found for every~$\degree$.

\begin{algorithm}\label{alg-main}
\caption{Given a natural number $\degree$, finds the smallest $n$ 
such that the full $\degree$-ary tree of depth $n$ satisfies 
the premises of Theorem~\ref{theo-main-tree}.} 
\label{alg-unif-distr-fps}
\begin{algorithmic}[1]
\State Let $n = 0$ and $T_n^{\degree}$ a full $\degree$-tree of depth $n$ 
rooted in $w$
\State $min = 1$ \Comment{Lower bound on the cardinality of a basis in 
$T_1^{\degree}$}
\State $max = \degree - 1$ \Comment{Upper 
bound on the cardinality of a basis in $T_1^{\degree}$}
\Repeat\label{st-repeat}
	\For{$i = min$ \textbf{to} $max$} \Comment{Each of these iterations can be 
	ran in parallel}
		\State Let $B$ be an empty set
		\ForAll {$S \subseteq V(T_n^{\degree})$ s.t. $|S| = i$}
			\If {$S$ is a resolving set}
				\If {$\forall_{S' \in B} 
				\multiplicity{\msrepr_{T_{n}^{\degree}}(w |
S')}{\ecc_{T_{n}^{\degree}}(w)} \neq \multiplicity{\msrepr_{T_{n}^{\degree}}(w
|
S)}{\ecc_{T_{n}^{\degree}}(w)}$}
					\State $B = B \cup \{S\}$
				\EndIf
			\EndIf
		\EndFor
		\If {$B \neq \emptyset $}	
			\State \textbf{break} \Comment{The \weak{} multiset dimension of 
			$T_n^{\degree}$ has 
			been found}
		\EndIf
	\EndFor
	\State \label{step-bound} $min = \dimms(T_n^{\degree}) \times \degree$ 
	\Comment{See 
	Lemma~\ref{theo-recursive-basis1}}
	\State $max = min + \degree-1$ \Comment{This is the trivial upper bound}
	\State $n = n + 1$
\Until{$|B| \geq \degree + 1$}\label{st-until} 
\State \Return $n$
\end{algorithmic}
\end{algorithm} 

Despite the exponential computational complexity of Algorithm~\ref{alg-main}, 
it terminates for $\degree = 2$. In this case, the smallest $n$ satisfying the 
premises of
Theorem~\ref{theo-main-tree} is $n = 4$. The three \orderlessBases{} $S_1$,
$S_2$ and $S_3$ of $T_4^2$ are illustrated in
Table~\ref{table-bases} below\footnote{Our program took
about $3.25$ hours in a
DELL computer with
processor i7-7600U and installed memory 16GB to find the result shown in
Table~\ref{table-bases}. }. We refer
the interested reader to Appendix~\ref{appendix} for a visual
representation of the bases shown in Table~\ref{table-bases}. The main
corollary of this result is the following.
\begin{corollary}
The \orderlessMetricDimension{} of a full $2$-ary tree $T_{\ell}^2$ of depth
$\ell$ is:
\[
\dimms(T_\ell^2) =
\begin{cases}
    1,& \text{if } \ell = 1\\
    3,& \text{if } \ell = 2\\
    6,& \text{if } \ell = 3\\
    13,& \text{if } \ell = 4\\
    2^{\ell-4} \times 13,              & \text{otherwise}
\end{cases}
\]
\end{corollary}

\begin{proof}
The first four cases are calculated by an exhaustive search using a computer
program that can be found at \url{https://github.com/rolandotr/graph}. The last
case follows from Theorem~\ref{theo-main-tree}.
\end{proof}

\begin{table}
\caption{Three \orderlessBases{} of $T_4^2$ satisfying the premises of
Theorem~\ref{theo-main-tree}. Vertices of $T_4^2$ have been labelled by using
a breadth-first ascending order, starting by labelling the root node with
$1$ and finishing with the label $2^{n+1}-1$.
}\label{table-bases}
\begin{center}
\begin{tabular}{r | l}
    \hline
    \multirow{5}{*}{} & $S_1 = \{22, 24, 14, 25, 26, 16, 28, 18, 2, 8, 30,
    20, 21\}$\\
	$T_4^2$ & $S_2 =\{22, 12, 24, 14, 26, 16, 28, 18, 6, 8, 30, 20, 21\}$\\
	& $S_3 = \{22, 24, 14, 25, 26, 16, 17, 28, 18, 8, 30, 20, 21\}$\\
    \hline

\end{tabular}
\end{center}
\label{tab:multicol}
\end{table}

It is worth remarking that Algorithm~\ref{alg-main} can be paralellised  and 
hence benefit from a computer cluster. Running the algorithm in a high 
performance computing facility is thus part of future work, which may lead to 
termination of Algorithm~\ref{alg-main} for values of $\degree$ higher than 
$2$.

\section{Conclusions}

In this paper we have addressed the problem of uniquely characterising
vertices in a graph by means of their multiset metric representations.
We have generalised the traditional notion of resolvability in such a way
that the new formulation allows for different structural characterisations
of vertices, including as particular cases the ones previously proposed
in the literature. We have pointed out a fundamental limitation
affecting previously proposed resolvability parameters
based on the multiset representation,
and have introduced a new notion of resolvability, the 
\orderlessMetricDimension,
which effectively addresses this limitation. Additionally, we have conducted
a study of the new parameter, where we have analysed its general behaviour,
determined its exact value for several graph families,
and proven the NP-hardness of its computation, while providing
an algorithm that efficiently handles some particular cases. 

\vspace*{.2cm}
\noindent \textbf{Acknowledgements:} The work reported in this paper 
was partially funded by Luxembourg's Fonds National de la Recherche (FNR), 
via grants C15/IS/10428112 (DIST) and C17/IS/11685812 (PrivDA). 


\newpage
\section*{Appendix}\label{appendix}

Here, the reader can find graphical representations for different 
outer multiset bases in a full $2$-ary tree of depth 
$4$. In the figures, a basis is formed by the red-coloured vertices.  

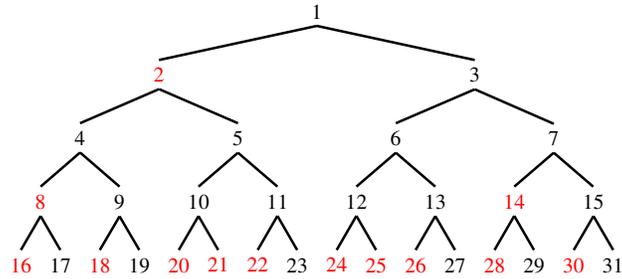
\begin{figure}[h]
	\centering
\begin{tikzpicture}[scale=0.8]
\tikzset{edge from parent/.append style={very thick}}
\Tree [.1 
		[.\textcolor{red}{2} 
			[.4 
				[.\textcolor{red}{8} 
					[.\textcolor{red}{16} ]
					[.17 ]
				] 
				[.9 
					[.\textcolor{red}{18} ]
					[.19 ]
				] 
            ] 
            [.5 
				[.10 
					[.\textcolor{red}{20} ]
					[.\textcolor{red}{21} ]
				] 
				[.11 
					[.\textcolor{red}{22} ]
					[.23 ]
				] 
			] 
		]
		[.3
			[.6 
				[.12 
					[.\textcolor{red}{24} ]
					[.\textcolor{red}{25} ]
				] 
				[.13 
					[.\textcolor{red}{26} ]
					[.27 ]
				] 
            ] 
            [.7 
				[.\textcolor{red}{14} 
					[.\textcolor{red}{28} ]
					[.29 ]
				] 
				[.15 
					[.\textcolor{red}{30} ]
					[.31 ]
				] 
			] 
		]
	 ]
\end{tikzpicture}
\caption{The multiset 
representation 
of the root vertex with respect to the set of red-coloured vertices is $\{1, 2^2, 
4^{10}\}$.}\label{fig-base1}
\end{figure}

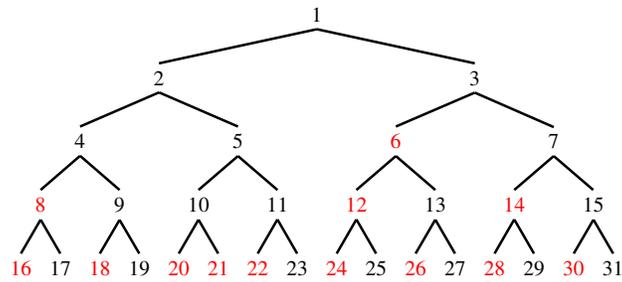
\begin{figure}[h]
	\centering
\begin{tikzpicture}[scale=0.8]
\tikzset{edge from parent/.append style={very thick}}
\Tree [.1 
		[.2
			[.4 
				[.\textcolor{red}{8} 
					[.\textcolor{red}{16} ]
					[.17 ]
				] 
				[.9 
					[.\textcolor{red}{18} ]
					[.19 ]
				] 
            ] 
            [.5 
				[.10 
					[.\textcolor{red}{20} ]
					[.\textcolor{red}{21} ]
				] 
				[.11 
					[.\textcolor{red}{22} ]
					[.23 ]
				] 
			] 
		]
		[.3
			[.\textcolor{red}{6} 
				[.\textcolor{red}{12} 
					[.\textcolor{red}{24} ]
					[.25 ]
				] 
				[.13 
					[.\textcolor{red}{26} ]
					[.27 ]
				] 
            ] 
            [.7 
				[.\textcolor{red}{14} 
					[.\textcolor{red}{28} ]
					[.29 ]
				] 
				[.15 
					[.\textcolor{red}{30} ]
					[.31 ]
				] 
			] 
		]
	 ]
\end{tikzpicture}
\caption{The multiset 
representation 
of the root vertex with respect to the set of red-coloured vertices is $\{2, 3^3, 
4^{9}\}$.}\label{fig-base2}
\end{figure}

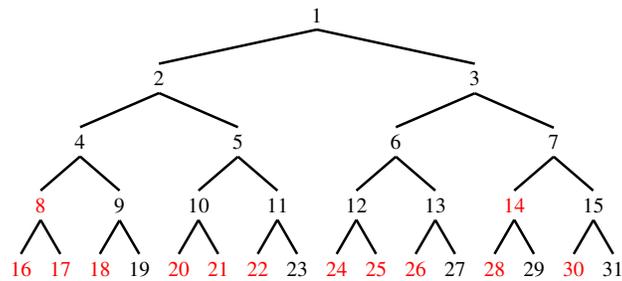
\begin{figure}[h]
	\centering
\begin{tikzpicture}[scale=0.8]
\tikzset{edge from parent/.append style={very thick}}
\Tree [.1 
		[.2
			[.4
				[.\textcolor{red}{8} 
					[.\textcolor{red}{16} ]
					[.\textcolor{red}{17} ]
				] 
				[.9 
					[.\textcolor{red}{18} ]
					[.19 ]
				] 
            ] 
            [.5 
				[.10
					[.\textcolor{red}{20} ]
					[.\textcolor{red}{21} ]
				] 
				[.11 
					[.\textcolor{red}{22} ]
					[.23 ]
				] 
			] 
		]
		[.3
			[.6 
				[.12 
					[.\textcolor{red}{24} ]
					[.\textcolor{red}{25} ]
				] 
				[.13 
					[.\textcolor{red}{26} ]
					[.27 ]
				] 
            ] 
            [.7 
				[.\textcolor{red}{14} 
					[.\textcolor{red}{28} ]
					[.29 ]
				] 
				[.15 
					[.\textcolor{red}{30} ]
					[.31 ]
				] 
			] 
		]
	 ]
\end{tikzpicture}
\caption{The multiset 
representation 
of the root vertex with respect to the set of red-coloured vertices is $\{3^2, 
4^{11}\}$. 
\label{fig-base3}}
\end{figure}

\end{document}